\documentclass{amsart}
\makeatletter
\@namedef{subjclassname@2020}{%
  \textup{2020} Mathematics Subject Classification}
\makeatother 

\usepackage{amsthm,amssymb,amsfonts,latexsym,mathtools,thmtools, lscape}
\usepackage[dvipsnames]{xcolor} 
\usepackage[T1]{fontenc}
\usepackage{mathrsfs}
\usepackage{amsmath} 
\usepackage{algorithm}
\usepackage{multirow}
\usepackage{tikz-cd} 
\usepackage{enumitem} 
\usepackage{hyperref} 
\hypersetup{
    colorlinks=true,
    linkcolor=blue,
    filecolor=blue,      
    urlcolor=cyan,
    linktocpage=true
}
\usepackage{algorithm}
\usepackage{multicol}
\usepackage{algpseudocode}
\usepackage{caption}
\usepackage{etoolbox}

\usepackage{verbatim} 
\usepackage{fancyvrb}
\usepackage{xcolor}
\usepackage{pb-diagram} 
\usepackage{supertabular}
\usepackage{longtable}

\newtheorem{theorem}{Theorem}[section]

\theoremstyle{definition}
\newtheorem{definition}[theorem]{Definition}
\newtheorem{example}[theorem]{Example}

\newtheorem{remark}[theorem]{Remark}

\newtheorem{proposition}[theorem]{Proposition}

\numberwithin{equation}{section}

\begin{document}

\title[Finite GS bases for double extension of type (14641)]{Finite Gr\"obner-Shirshov bases for double extension regular algebras of type (14641)}



\author{Karol Herrera}
\address{Universidad ECCI}
\curraddr{Campus Universitario}
\email{kherrerac@ecci.edu.co}
\thanks{}

\author{Sebasti\'an Higuera}
\address{Universidad ECCI}
\curraddr{Campus Universitario}
\email{shiguerar@ecci.edu.co}

\author{Andr\'es Rubiano}
\address{Universidad ECCI}
\curraddr{Campus Universitario}
\email{arubianosecci.edu.co}

\thanks{}


\date{}
\dedicatory{Dedicated to our dear friend Fabio Calder\'on} 

\begin{abstract}

In this paper, we compute the Gr\"obner–Shirshov bases for certain regular double extension algebras by means of an algorithm implemented in \texttt{Matlab}, which facilitates the underlying algebraic computations. Moreover, we establish that these families also admit a PBW basis.
\end{abstract}

\maketitle

\section{Introduction}

Shirshov \cite{Shirshov1962} (originally in Russian) introduced the first notion of what is now called {\em Gr\"obner–Shirshov bases} (GS bases, for short). He posed the following question: how can one find a linear basis of any Lie algebra presented by generators and defining relations? To address this, he developed an infinite algorithm (currently known as {\em Shirshov’s algorithm} or the reduction algorithm) based on the new concept of {\em composition} of two Lie polynomials, together with {\em Shirshov’s composition lemma for Lie algebras}. Importantly, this lemma is also valid for associative algebras, a fact that Shirshov recognized from the very beginning. Consequently, {\em the composition lemma for associative algebras} was formulated years later by his student Bokut in \cite{Bokut1976}, motivated by the need to handle more intricate defining relations in associative algebras. In \cite{Bokut1976}, Bokut also cites a preprint of Bergman \cite{Bergman1978} (later published), where the diamond lemma was established for ring theory. Thus, the composition lemma and the diamond lemma are essentially equivalent, differing only in terminology. In fact, the diamond lemma was formulated and proved for commutative, noncommutative, and Lie polynomials, and eventually both results came to be referred to under a unified name: the {\em Diamond–Composition lemma}.

In the case of commutative algebras, GS bases correspond to the Gr\"obner bases introduced by Buchberger in his Ph.D. thesis \cite{Buchberger1965,Buchberger1970}. Specifically, the notion of composition corresponds to the $S$-polynomials, while the concept of a set of polynomials that is \textquotedblleft closed\textquotedblright \ under composition—that is, the GS definition of bases—coincides, in the commutative case, with Buchberger’s criterion. Moreover, one implication of the Diamond–Composition lemma asserts that if a set $S$ is closed under composition and $f \in {\rm Id}(S)$, then the leading word $\overline{f}$ has the form $u\overline{s}v$ for some $s \in S$. In the commutative setting, this statement is equivalent to the definition of Gr\"obner bases. Both Shirshov and Buchberger constructed algorithms to compute GS bases and Gr\"obner bases, respectively, and the underlying ideas are essentially the same: iteratively adding to a set $S$ all nontrivial compositions (respectively, $S$-polynomials) under reduction with respect to $S$ until a closed set (a Gr\"obner basis) is obtained. Shirshov denoted the closure of $S$ by $S^*$. However, in contrast to Buchberger’s algorithm, which always terminates, Shirshov’s algorithm does not necessarily do so.

On the other hand, {\em Artin–Schelter regular algebras} (AS-regular for short), introduced by Artin and Schelter \cite{ArtinSchelter1987}, can be regarded as noncommutative analogues of polynomial rings. By definition, an Artin–Schelter regular algebra is an $\mathbb{N}$-graded algebra $A = \bigoplus_{n \geq 0} A_n$ over a field $\Bbbk$ that is connected ($A_0 = \Bbbk$) and satisfies the following conditions:
\begin{enumerate}
\item[\rm (i)] $A$ has finite global dimension $d$;
\item[\rm (ii)] $A$ has polynomial growth; and
\item[\rm (iii)] ({\em Gorenstein condition}) $A$ is Gorenstein, in the sense that ${\rm Ext}_A^{i}(\Bbbk, A) = 0$ for $i \neq d$, and ${\rm Ext}_A^{d}(\Bbbk, A) \cong \Bbbk$.
\end{enumerate}

So far, the classification of AS-regular algebras of global dimension $1$, $2$, $3$, and $4$ generated in degree $1$ is as follows:
\begin{itemize}
    \item \cite[Example 1.8]{Rogalski2023}: If $A$ is a $\Bbbk$-algebra that is regular of global dimension one, then $A \cong \Bbbk[x]$.
    \item \cite[Example 1.9]{Rogalski2023}: If $A$ is a regular algebra of global dimension two, then either $A \cong A_q = \Bbbk{x, y}/\langle yx - qxy\rangle$ for some $0 \neq q \in \Bbbk$ (the {\em Quantum plane}), or $A \cong A_J = \Bbbk{x, y}/\langle yx - xy - x^2\rangle$ (the {\em Jordan plane}).
    \item \cite[Lemma 2.3.1]{bellamy2016noncommutative}: If $A$ is an AS-regular algebra of global dimension three generated in degree $1$, then exactly one of the following holds:
    \begin{enumerate}
        \item[\rm (1)] $A \cong \Bbbk\langle t_1,t_2,t_3\rangle/(f_1,f_2,f_3)$, where each $f_i$ has degree $2$ for $i=1,2,3$, and $h_A(t) = \frac{1}{(1-t)^3}$.
        \item[\rm (2)] $A \cong \Bbbk\langle t_1,t_2\rangle/(f_1,f_2)$, where each $f_i$ has degree $3$ for $i=1,2$, and $h_A(t) = \frac{1}{(1-t)^2(1-t^2)}$.
    \end{enumerate}
    \item In \cite{SmithStafford1992}, Smith and Stafford studied the first AS-regular algebras of global dimension four, namely the four-dimensional Sklyanin algebras. Later, Zhang and Zhang \cite{ZhangZhang2008} proved that a graded and connected double Ore extension of a regular AS algebra is also AS-regular. Building on this result, they constructed 26 families of AS-regular algebras of global dimension four, whose relationships were later described by Rubiano and Reyes in \cite{RubianoReyes2024}. The classification of AS-regular algebras in higher dimensions remains an open and active area of research.
\end{itemize}

 Recall that for an associative unital ring $R$, an endomorphism $\sigma$ of $R$, and a $\sigma$-derivation $\delta$ of $R$, the {\em Ore extension} (or {\em skew polynomial ring}) of $R$ is obtained by adjoining a generator $x$ to $R$ subject to the relation $xr = \sigma(r)x + \delta(r)$ for all $r \in R$. This Ore extension of $R$ is denoted by $R[x; \sigma, \delta]$. Zhang and Zhang \cite{ZhangZhang2008, ZhangZhang2009} introduced a generalization of Ore extensions, called {\em double Ore extensions} (or simply {\em double extensions}). This notion arises as a natural enlargement of the classical construction, and its definition exhibits similarities with the two-step iterated process of successive Ore extensions. However, unlike Ore extensions, double extensions retain very few of the same structural properties, since many of the techniques applicable in the former fail in the latter \cite[Section 4]{ZhangZhang2008}. Indeed, Zhang and Zhang emphasize that the study of double extensions is substantially more intricate, often requiring additional restrictive conditions, and their general ring-theoretic behavior remains largely unknown \cite[Section 0]{ZhangZhang2008}. Despite these challenges, their connections with other algebraic structures have been actively investigated, including Poisson, Hopf, Koszul, and Calabi–Yau algebras (see, e.g., \cite{GomezSuarez2020, Li2022, LouOhWang2020, LuOhWangYu2018, LuWangZhuang2015, RamirezReyes2024, SuarezLezamaReyes2017, ZhuVanOystaeyenZhang2017}).

In their article \cite{ZhangZhang2009}, Zhang and Zhang studied regular algebras $B$ of dimension four generated in degree one. The projective resolution of the trivial module $\Bbbk_B$ for such an algebra takes the form
\begin{equation}\label{LISTResolution}
0 \longrightarrow B(-4) \longrightarrow B(-3)^{\oplus 4} \longrightarrow B(-2)^{\oplus 6} \longrightarrow B(-1)^{\oplus 4} \longrightarrow B \longrightarrow \Bbbk_B \longrightarrow 0.
\end{equation}

Based on this resolution, they referred to these algebras as being of type (14641). Zhang and Zhang provided a classification of all double extensions $R_P[y_1, y_2;\sigma]$ (with $\delta = 0$ and $\tau = (0,0,0)$) of type (14641). Taking into account that Ore extensions and normal extensions of regular algebras of dimension three had already been studied by Le Bruyn et al. \cite{LeBruynSmithVandenBergh1996}, they excluded some of these from their analysis. Consequently, their “partial” classification comprises 26 families of regular algebras of type (14641), labeled by $\mathbb{A}, \mathbb{B}, \dotsc, \mathbb{Z}$.

This paper is organized as follows. In Section \ref{preliminares}, we introduce the fundamental concepts of ambiguity and trivial composition, together with the notion of a GS basis. The composition algorithm for computing such bases is described and illustrated with examples. We also establish necessary and sufficient conditions for an algebra defined by relations to admit a PBW basis in terms of GS bases. Moreover, we define the algebras that are central to this study: regular double extension algebras of type (14641), and examine their most relevant structural properties. In Section \ref{Matlab}, we present a Matlab implementation of part of Shirshov's algorithm, which reduces compositions modulo a set $S$, following the approach developed in \cite{Herrera2024}. A small example involving a free algebra generated by three variables and defined by three relations is provided. Finally, in Section \ref{GS bases for 14641}, we compute finite GS bases for 12 families of regular double extension algebras of type (14641), out of the 26 described in \cite{RubianoReyes2024}, and verify that these algebras admit a PBW basis.

Throughout the paper, let $\Bbbk$ be a field, $X$ an alphabet, and $\Bbbk\{X\}$ the free associative algebra over $\Bbbk$ generated by $X$, i.e., the algebra of noncommutative polynomials in $X$ with coefficients in $\Bbbk$. We also denote by $X^*$ the free monoid, that is, the set of all words in $X$.

\section{Preliminaries}\label{preliminares}
We begin by studying the theory of GS bases for free algebras and PBW bases, presenting the relevant definitions, results, and an illustrative example of how such bases are defined and computed. We then introduce regular double extension algebras of type (14641), highlighting some of their structural properties and providing examples.

\subsection{Gr\"obner-Shirshov bases for free algebras}\label{Section:GS-bases theory}

 The algebra $\Bbbk\{ X \}$  satisfies a following universal property: for each $\Bbbk$-algebra $A$ and each function $\epsilon : X \to A$, there exists a unique homomorphism of $\Bbbk$-algebras, $f: \Bbbk\{ X\} \to A $ such that $f \circ \iota=\epsilon$:
\[
\begin{diagram}
\node{X} \arrow{e,t}{\iota} \arrow{s,l}{\epsilon}
\node{\Bbbk\{ X\}}  \arrow{sw,b,..}{f}\\
\node{A} 
\end{diagram}
\]
\[
f(r\cdot x):=r\cdot\epsilon(x), \ r\in \Bbbk,\ x \in X.
\]

It follows that any associative algebra $A$ is a quotient of some $\Bbbk\{ X \}$, 
\[
A=\Bbbk\{ X \}/I, \hspace{0.1cm} \text{where $I$ is an ideal of $\Bbbk\{ X \}$ }.
\]
 
Note that the set $X^*$ is a linear basis of  $\Bbbk\{ X \}$. Now, given a set $S$ of generators of $I$, we denote by ${\rm Id}(S):=\left\{ \sum_i \alpha_ia_is_ib_i \mid \ a_i,b_i \in X^*,\ \alpha_i \in \Bbbk, \ s_i\in S\right\}$ as the {\em two-sided ideal} of $\Bbbk\{ X \}$ generated by $S$. Then $\Bbbk\{ X \}/{\rm Id}(S)$ is the associative algebra with generators $X$ and defining relations $s_i=0$, $s_i\in S$.
 
There are different \textquotedblleft rules\textquotedblright \ or ways to order the set $X^*$, these are called monomial orders. More precisely, a {\em monomial ordering} $\preceq$ on $X^*$ is well ordering such that it is  compatible with the multiplication of words, that is, for $u,v \in X^*$, we have $u \preceq v$ implies $w_1uw_2 \preceq w_1vw_2$, for all $w_1,w_2 \in X^*$. For example, the \texttt{lexicographic ordering}  denoted by $\prec_{\rm \texttt{lex}}$ is well ordered on $X^*$, but this is not a monomial ordering because it is not compatible with the multiplication of words: $x_1^j \prec_{\rm \texttt{lex}} x_1^{j+1}$ but $x_1^{j+1}x_2\prec_{\rm \texttt{lex}} x_1^jx_2$, for some $j\geq1$. If we consider the notion of {\em degree} in the above ordering, we can be reformulated to obtain a monomial ordering in $X^*$. We  denote  $\prec_{\texttt{deglex}}$ or simply $\prec$  the \texttt{degree lexicographical ordering} on $X^*$. For example, if  $x_1\prec x_2\prec x_3$, then $x_2x_3^2x_1^3\succ_{\texttt{deglex}} x_3^2x_1^2$, since $| x_2x_3^2x_1^3 |=6 > 4=|x_3^2x_1^2 |$ or  $u=x_3^2x_2x_3^3x_1\succ_{\texttt{deglex}} x_3x_2^3x_1^2x_3=v$, since $|u|=7=|v|$ and $u\succ_{\texttt{lex}} v$. 

Now, given a polynomial $f\in \Bbbk\{ X\}$, the \textit{leading monomial} denoted by $\overline{f}\in X^*$ is such that $f=\alpha_{\overline{f}}\overline{f}+ \sum\alpha_iu_i$, with $0 \neq \alpha_{\overline{f}},\ \alpha_i \in \Bbbk,\ u_i\in X^*$ and every $u_i\prec \overline{f}$. The element $\alpha_{\overline{f}}$ is called \textit{the leading coefficient of} $f$ and the polynomial $f$ is said to be \textit{monic} if $\alpha_{\overline{f}}=1$. 

\begin{definition}[{\cite[p. 37]{BokutChen2014}}]
Let $f$ and $g$ be two monic polynomials in $\Bbbk\{ X \}$ and $(X^*,\preceq)$ be a monomial ordering. 
 \begin{enumerate}
     \item [(i)] If $w$ is a word such that $w =\overline{f}a = b\overline{g}$ for some $a,b \in X^*$ with $|\overline{f}|+|\overline{g}| > |w|$, then the polynomial $(f, g)_w := fa-bg$ is called the \textit{intersection composition} of $f$ and $g$ with respect to $w$.
     \item[(ii)] If $w =\overline{f} = a\overline{g}b$ for some $a,b \in X^*$, then the polynomial $(f, g)_w:= f-agb$ is called the \textit{inclusion composition} of $f$ and $g$ with respect to $w$.
 \end{enumerate}
 In the composition $(f, g)_w$, $w$ is called an \textit{ambiguity} or the ${\rm lcm}(\overline{f},\overline{g})$.
 \label{Def.Composition}
\end{definition}

The following result shows when compositions cannot be formed with the same polynomial

\begin{proposition}[{\cite[Lemma 2]{Shirshov1962}}]\label{composition(ff)} 
Fix a monomial ordering on $(X^*,\preceq)$. Let $f\in \Bbbk \{ X\}$ a polynomial such that $\overline{f}$ is regular. Then the composition $(f,f)_w$ cannot be formed.
\end{proposition}

When computing the ambiguity between two polynomials, three cases may arise: it does not exist, it is unique, or there are at least two possibilities. If no ambiguity exists, it is evident that no composition can be formed, i.e., there are no compositions.

\begin{definition}[{\cite[p. 37]{BokutChen2014}}]
Given a non-empty subset $S \subset \Bbbk\{ X \}$, we say that the composition $(f,g)_w$ is \textit{trivial} module $S$ if 
 \[ 
 (f,g)_w=\sum\alpha_ia_is_ib_i \in {\rm Id}(S),
 \]
 where each $\alpha_i \in \Bbbk$, $a_i,b_i \in X^*$, $s_i\in S$ and $a_i\overline{s_i}b_i \prec w$. The elements $a_is_ib_i$ are called \textit{$S$-words}. In this case, we write $(f,g)_w \equiv 0\ {\rm mod}(S)$.
\end{definition}
To determine if a polynomial belongs to an ideal, we will use the \textit{reduction systems}, which is nothing more than a generalization of the algorithm of the division in several determinates, but for the non-commutative case. 

\begin{definition}[{\cite[Definition 1.1]{ReyesSuarezMomento2017}}]
\begin{itemize}
    \item [(i)] Let $X$ be a non-empty set. A subset $Q\subseteq X^*\times \Bbbk\{ X\}$ is called a \textit{reduction system} for $\Bbbk\{ X\}$. An element $\sigma=(W_\sigma,f_\sigma)\in Q$ has components $W_\sigma$ a word in $X^*$ and $f_\sigma$ a polynomial in $\Bbbk\{ X\}$. Note that every reduction system for $\Bbbk\{ X\}$ defines a factor ring $\Bbbk\{ X\} / { \rm Id}(Q)$, with ${ \rm Id}(Q)$ the two-sided ideal of $\Bbbk\{X\}$ generated by the polynomials $W_\sigma-f_\sigma$, with $\sigma\in Q$.
    
    \item[(ii)] If $\sigma$ is an element of a reduction system $Q$ and $a,b\in X^*$, the $\Bbbk$-linear endomorphism $r_{a\sigma b}: \Bbbk\{ X \} \to \Bbbk\{ X \}$, which fixes all elements in the basis $X^*$ different from $aW_{\sigma}b$ and sends this particular element to $af_{\sigma}b$ is called a \textit{reduction} for $Q$. If $r$ is a reduction and $f\in \Bbbk \{ X\}$, then $f$ and $r(f)$ represent the same element in the $\Bbbk$-ring $\Bbbk\{ X\} / { \rm Id}(Q)$. Thus, reductions may be viewed as rewriting rules in this factor ring.
    
    \item[(iii)] A reduction $r_{a\sigma b}$ acts trivially on an element $f\in \Bbbk\{ X\}$ if $r_{a\sigma b}(f)=f$. An element $f\in \Bbbk\{ X\}$ is said to be \textit{irreducible} under $Q$ if all reductions act trivially on $f$.
    
    \item[(iv)] Let $f$ be an element of $\Bbbk\{ X\}$. We say that $f$ \textit{reduces} to $g \in \Bbbk\{ X\}$ if there is a finite sequence $r_1,\dots, r_n$ of reductions such that $g=(r_1,\dots, r_n)f$. We will write $f\equiv g \ {\rm mod}(Q)$. A finite sequence of reductions $r_1,\dots, r_n$ is said to be final on $f$, if $(r_1,\dots, r_n)f\in {\rm Irr}(Q)$.
\end{itemize}
\end{definition}
Note that a polynomial does not belong to an ideal if its reduction is irreducible. 

\begin{definition}[{\cite[Definition 4.1]{BokutChen2014}}]
 Let $S\subset \Bbbk\{ X \} $ be a non-empty set of monic polynomials. Fix a monomial ordering $\preceq$ on $X$. $S$ is called a \textit{Gr\"obner-Shirshov basis} of the ideal ${\rm Id}(S)$ of $\Bbbk\{ X \}$ with respect to $(X^*,\preceq)$ if any composition $(f,g)_w$, with $f,g\in S$, is trivial modulo $S$.
 \label{Def.GSB}
\end{definition}

Given $S \subset \Bbbk\{ X\}$ be a GS basis of the ideal ${\rm Id}(S)$ we denote the set
\[
{\rm Irr}(S):=\{ u\in X^*\mid u\neq a\overline{s}b, s\in S, a,b\in X^*\}
\]
as  all monomials that do not have $\overline{s}$ as a subword, for all $s\in S$. The elements of ${\rm Irr}(S)$ are called \textit{$S$-irreducible or $S$-reduced}.  

Now, let us recall the main result of this theory, commonly referred to as the Diamond–Composition Lemma.

\begin{proposition}[{\cite[Theorem 4.4]{BokutChen2014}; \cite[Proposition 1]{Bokut1976}}]\label{CDL}
Choose a monomial ordering $\preceq$ on $X^*$. Consider a monic set $S \subset \Bbbk \{ X\}$ and the ideal ${\rm Id}(S)\subseteq \Bbbk \{ X\}$ generated by $S$. The following statements are equivalent:
\begin{itemize}
    \item [{\rm (1)}] $S$ is a Gr\"obner-Shirshov basis in $\Bbbk \{ X\}$.
    \item[{\rm (2)}] If $f\in {\rm Id}(S)$, then $\overline{f}=a\overline{s}b$ for some $s\in S$ and $a,b\in X^*$.
    \item[{\rm (3)}] ${\rm Irr}(S)=\{ u\in X^*\mid u\neq a\overline{s}b, s\in S, a,b\in X^*\}$ is a linear basis of the algebra $\Bbbk \{ X\}/ {\rm Id}(S)$.
\end{itemize}
\end{proposition}

To construct a GS basis, we imitate Buchberger’s algorithm by extending the original set. We denote by $S^c$ the \textit{completion} of $S$, that is, the set obtained by successively adding all non-trivial compositions through Shirshov’s (or reduction) algorithm until a GS basis is reached. Note that $S^c$ contains $S$ and generates the same ideal, i.e., ${\rm Id}(S^c) = {\rm Id}(S)$, since each added non-trivial composition in $S^c$ already belongs to ${\rm Id}(S)$.

\begin{algorithm}[t]
\caption{Shirshov’s algorithm}
\begin{algorithmic}[1]
\State ${\rm Input}: S=\{s_1,s_2,\dots,s_t\} \subseteq \Bbbk\{ X\}$
\State ${\rm Output:  \text{a GS basis}}\  S^c=\{s_{c_1},s_{c_2},\dots,s_{c_k}\} \ {\rm for} \ {\rm Id}(S), {\rm with } \ S\subseteq S^c$
\State $S^c:=S$
\Repeat
  \State $S':=S^c$
  \For{each pair $\{f,g\}$, $f\neq g$ in $S'$}
    \State $r:=(f,g)_w$ and $r':=(g,f)_{w'}$
    \If{$r\not\equiv 0 \ {\rm mod}(S')$ \ {\rm or} $r'\not\equiv 0 \ {\rm mod}(S')$ } 
      \State $S^c:=S^c \cup \{r\}$ \Comment{or $S^c:=S^c \cup \{r'\}$ }
    \EndIf 
  \EndFor
\Until{$S^c=S'$}
\State \textbf{return} $S^c$  \Comment{this can have infinite steps}
\end{algorithmic}
\label{AlgorithmGS}
\end{algorithm}

We are going to present an illustrative example Shirshov's algorithm. Recall that we are considering the \texttt{degree lexicographical order} in $X^*$.

\begin{example}\label{ExampleSklyanin}
Let us consider the polynomials $f_1=x^2-\frac{1}{s}yz+\frac{a}{s}zy$, $f_2=xy-ayx-sz^2$ and $f_3=xz-\frac{1}{a}zx+\frac{s}{a}y^2$ in $\Bbbk\{ x,y,z\}$ where $a,s\neq 0$ and $a^3=-1$.  Note that the leading terms of these polynomials are $\Bar{f_1}=x^2$, $\Bar{f_2}=xy$ and $\Bar{f_3}=xz$ with  $z\prec y\prec x$. Let us see that $S=\{f_1,f_2,f_3\}$ is not a GS basis of the ideal  ${\rm Id}(S)$. The set $S$ only admits three possible compositions with respect to $w_1,w_2$ and $w_3$, and only one is trivial:
{\small{\begin{align*}
    w_1&=\overline{f_1}(a)=(b)\overline{f_2}, \ {\rm with} \ |w_1|<|\overline{f_1}|+|\overline{f_2}|=4 \ {\rm and} \ a,b\in X^*\\
    &=(x^2)( y )=( x )(xy)\\
    &=x^2y.\\
    (f_1,f_2)_{w_1}&=\left(x^2-\frac{1}{s}yz+\frac{a}{s}zy \right)y-x(xy-ayx-sz^2)\\
    &=-\frac{1}{s}yzy+\frac{a}{s}zy^2+axyx+sxz^2\\
    &\equiv-\frac{1}{s}yzy+\frac{a}{s}zy^2+a(ayx+sz^2)x+s\left( \frac{1}{a}zx-\frac{s}{a}y^2\right)z\\
    &=-\frac{1}{s}yzy+\frac{a}{s}zy^2+a^2yx^2+asz^2x+\frac{s}{a}zxz-\frac{s^2}{a}y^2z\\
    &\equiv -\frac{1}{s}yzy+\frac{a}{s}zy^2+a^2y\left( \frac{1}{s}yz-\frac{a}{s}zy\right)+asz^2x+\frac{s}{a}z\left( \frac{1}{a}zx-\frac{s}{a}y^2\right)\\
    &\ \ -\frac{s^2}{a}y^2z\\
    &=-\frac{1}{s}yzy+\frac{a}{s}zy^2+\frac{a^2}{s}y^2z-\frac{a^3}{s}yzy+asz^2x+\frac{s}{a^2}z^2x-\frac{s^2}{a^2}zy^2-\frac{s^2}{a}y^2z\\
    &=-\frac{1}{s}yzy+\frac{a}{s}zy^2+\frac{a^2}{s}y^2z+\frac{1}{s}yzy+asz^2x+\frac{s}{a^2}z^2x-\frac{s^2}{a^2}zy^2-\frac{s^2}{a}y^2z\\
    \end{align*}}}
    {\small{\begin{align*}
    &=\left( as+\frac{s}{a^2}\right)z^2x+\left( \frac{a^2}{s}-\frac{s^2}{a}\right)y^2z+\left( \frac{a}{s}-\frac{s^2}{a^2}\right)zy^2\\
    &=\left( \frac{a^3s+s}{a^2}\right)z^2x+\left( \frac{a^3-s^3}{sa}\right)y^2z+\left( \frac{a^3-s^3}{sa^2}\right)zy^2\\
    &\equiv y^2z+\frac{1}{a}zy^2 \not\equiv 0 \ {\rm mod}(S).\\
    &\\
    w_2&=\overline{f_1}(a)=(b)\overline{f_3}, \ {\rm with} \ |w_2|<|\overline{f_1}|+|\overline{f_3}|=4 \ {\rm and} \ a,b\in X^*\\
    &=(x^2)( z )=( x )(xz)\\
    &=x^2z.\\
    (f_1,f_3)_{w_2}&=\left(x^2-\frac{1}{s}yz+\frac{a}{s}zy \right)z-x\left(xz-\frac{1}{a}zx+\frac{s}{a}y^2 \right)\\
    &=-\frac{1}{s}yzy+\frac{a}{s}zyz+\frac{1}{a}xzx-\frac{s}{a}xy^2\\
    &\equiv-\frac{1}{s}yzy+\frac{a}{s}zyz+\frac{1}{a}\left(\frac{1}{a}zx-\frac{s}{a}y^2 \right)x-\frac{s}{a}(ayx+sz^2)y\\
    &=-\frac{1}{s}yzy+\frac{a}{s}zyz+\frac{1}{a^2}zx^2-\frac{s}{a^2}y^2x-syxy-\frac{s^2}{a}z^2y\\
    &\equiv-\frac{1}{s}yzy+\frac{a}{s}zyz+\frac{1}{a^2}z\left( \frac{1}{s}yz-\frac{a}{s}zy\right)-\frac{s}{a^2}y^2x-sy(ayx+sz^2)-\frac{s^2}{a}z^2y\\
    &=-\frac{1}{s}yzy+\frac{a}{s}zyz+\frac{1}{a^2s}zyz-\frac{1}{as}z^2y-\frac{s}{a^2}y^2x-say^2x-s^2yz^2-\frac{s^2}{a}z^2y\\
    &=\left(-s^2-\frac{1}{s} \right)yz^2+\left( -\frac{s^2}{a}-\frac{1}{as}\right)z^2y+\left( \frac{a}{s}-\frac{1}{a^2s}\right)zyz-\left(\frac{s}{a^2}+sa \right)y^2x\\
    &=-\left(\frac{s^3+1}{s} \right)yz^2-\left(\frac{as^3+a}{a^2s} \right)z^2y+\left(\frac{a^3s+s}{a^2s} \right)zyz+\left(\frac{s+sa^3}{a^2} \right)y^2x\\
    &=-\left(\frac{s^3+1}{s} \right)yz^2-\left(\frac{as^3+a}{a^2s} \right)z^2y\\
    &=yz^2+\frac{1}{a}z^2y\not\equiv 0 \ {\rm mod}(S).\\
    &\\
    w_3&=\overline{f_1}(a)=(b)\overline{f_1}, \ {\rm with} \ |w_3|<|\overline{f_1}|+|\overline{f_1}|=4 \ {\rm and} \ a,b\in X^*\\
    &=(x^2)( x )=( x )(x^2)\\
    &=x^3.\\
    (f_1,f_1)_{w_3}&=\left( x^2-\frac{1}{s}yz+\frac{a}{s}zy \right)x-x\left( x^2-\frac{1}{s}yz+\frac{a}{s}zy\right)\\
    &=-\frac{1}{s}yzx+\frac{a}{s}zyx+\frac{1}{s}xyz-\frac{a}{s}xzy\\
    &\equiv-\frac{1}{s}yzx+\frac{a}{s}zyx+\frac{1}{s}(ayx+sz^2)-\frac{a}{s}\left(\frac{1}{a}zx-\frac{s}{a}y^2  \right)y\\
    &=-\frac{1}{s}yzx+\frac{a}{s}zyx+\frac{a}{s}yxz+z^3-\frac{1}{s}zxy+y^3\\
    &\equiv -\frac{1}{s}yzx+\frac{a}{s}zyx+\frac{a}{s}y\left( \frac{1}{a}zx-\frac{s}{a}y^2 \right)+z^3-\frac{1}{s}z(ayx+sz^2)+y^3\\
    &=-\frac{1}{s}yzx+\frac{a}{s}zyx+\frac{1}{s}yzx-y^3+z^3-\frac{a}{s}zyx-z^3+y^3\\
    &=0.
\end{align*}}}

Note that there are no more ambiguities, so there are no more compositions. Thus, the set $S$ is not a basis of GS for the ideal ${\rm Id}(S)$. Now, let us define a new set $S_1=S \cup \{ f_4,f_5 \}$ and check if this is a basis for GS, where $f_4= y^2z+\frac{1}{a}zy^2$ and $f_5=yz^2+\frac{1}{a}z^2y$. Note that the only possible compositions are when the final word is the beginning of the other word.
{\small{\begin{align*}
    w_4&=\overline{f_2}(a)=(b)\overline{f_4}, \ {\rm with} \ |w_4|<|\overline{f_2}|+|\overline{f_4}|=5 \ {\rm and} \ a,b\in X^*\\
    &=(xy)( yz )=( x )(y^2z)\\
    &=xy^2z.\\
    (f_2,f_4)_{w_4}&=(xy-ayx-sz^2)yz-x\left( y^2z+\frac{1}{a}zy^2\right)\\
    &=-ayxyz-sz^2yz-\frac{1}{a}xzy^2\\
    &\equiv -ay(ayx+sz^2)z-sz^2yz-\frac{1}{a}\left( \frac{1}{a}zx-\frac{s}{a}y^2 \right)y^2\\
    &=-a^2y^2xz-asyz^3-sz^2yz-\frac{1}{a}zxy^2+\frac{s}{a^2}y^4\\
    &\equiv -a^2y^2\left( \frac{1}{a}zx-\frac{s}{a}y^2\right)-as\left( -\frac{1}{a}z^2y\right)z-sz^2yz-\frac{1}{a^2}z(ayx+sz^2)y\\
    &\ \ +\frac{s}{a^2}y^4\\
    &=-ay^2zx+asy^4+sz^2yz-sz^2yz-\frac{1}{a}zyxy-\frac{s}{a^2}z^3y+\frac{s}{a^2}y^4\\
    &\equiv -a\left( -\frac{1}{a}zy^2\right)x+\left(  as+\frac{s}{a^2}\right)y^4-\frac{1}{a}zy(ayx+sz^2)-\frac{s}{a^2}z^3y\\
    &=zy^2x-\left(\frac{a^3s+s}{a^2}  \right)y^4-zy^2x-\frac{s}{a}zyz^2-\frac{s}{a^2}z^3y\\
    &\equiv-\frac{s}{a}z\left( z^2y \right)-\frac{s}{a^2}z^3y\\
    &=\frac{s}{a^2}z^3y-\frac{s}{a^2}z^3y\\
    &=0.\\
    &\\
    w_{5}&=\overline{f_2}(a)=(b)\overline{f_5}, \ {\rm with} \ |w_{5}|<|\overline{f_2}|+|\overline{f_5}|=5 \ {\rm and} \ a,b\in X^*\\
    &=(xy)( z^2 )=( x )(yz^2)\\
    &=xyz^2.\\
    (f_2,f_5)_{w_{5}}&=(xy-ayx-sz^2)z^2-x\left(yz^2+\frac{1}{a}z^2y \right)\\
    &=ayxz^2-sz^4-\frac{1}{a}xz^2y\\
    &\equiv -ay\left(\frac{1}{a}zx-\frac{s}{a}y^2  \right)z-sz^4-\frac{1}{a}\left(\frac{1}{a}zx-\frac{s}{a}y^2  \right)zy -yzxz+sy^3z-sz^4\\
    &\ \ -\frac{1}{a^2}zxzy+\frac{s}{a^2}y^2zy\\
    &\equiv -yz\left( \frac{1}{a}zx-\frac{s}{a}y^2\right)+sy\left( -\frac{1}{a}zy^2 \right)-sz^4-\frac{1}{a^2}z\left( \frac{1}{a}zx-\frac{s}{a}y^2 \right)y
     \end{align*}}}
    {\small{\begin{align*}
    &\ \ +\frac{s}{a^2}\left( -\frac{1}{a}zy^2 \right)y\\
    &=\frac{1}{a}yz^2x+
    -\frac{s}{a}yzy^2+\frac{s}{a}yzy^2-sz^4-\frac{1}{a^3}z^2xy+\frac{s}{a^3}zy^3-\frac{s}{a^3}zy^3\\
    &\equiv -\frac{1}{a}\left( -\frac{1}{a}z^2y \right)x-sz^4-\frac{1}{a^3}z^2\left(ayx+sz^2 \right)\\
    &=\frac{1}{a^2}z^2yx-sz^4-\frac{1}{a^2}z^2yx-\frac{s}{a^3}z^4\\
    &=-\left(\frac{sa^3+s}{a^3}  \right)z^4\\
    &=0.\\
    &\\
    w_{6}&=\overline{f_4}(a)=(b)\overline{f_5}, \ {\rm with} \ |w_{6}|<|\overline{f_4}|+|\overline{f_5}|=6 \ {\rm and} \ a,b\in X^*\\
    &=(y^2z)( z )=( y )(yz^2)\\
    &=y^2z^2.\\
    (f_4,f_5)_{w_{6}}&=\left( y^2z+\frac{1}{a}zy^2\right)z-y\left( yz^2+\frac{1}{a}z^2y \right)\\
    &=\frac{1}{a}zy^2z-\frac{1}{a}yz^2y\\
    &\equiv \frac{1}{a}z\left( -\frac{1}{a}zy^2\right)-\frac{1}{a}\left( -\frac{1}{a}z^2y \right)y\\
    &=-\frac{1}{a^2}z^2y^2+\frac{1}{a^2}z^2y^2\\
    &=0.
\end{align*}}}
we conclude that the set $S^c:=S_1$ is a Gr\"obner-Shirshov basis of the ideal ${\rm Id}(S)$

\begin{equation*}
    S^c=\Bigg\{x^2-\frac{1}{s}yz+\frac{a}{s}zy, \hspace{0.2cm} xy-ayx-sz^2,  \hspace{0.2cm} xz-\frac{1}{a}zx-\frac{s}{a}y^2,\hspace{0.2cm} y^2z+\frac{1}{a}zy^2,\hspace{0.2cm} yz^2+\frac{1}{a}z^2y\Bigg\}.
    \label{GS-2}
\end{equation*}
\end{example}

\subsection{PBW bases}

 Following Li \cite{Li2002}, a finitely generated $\Bbbk$-algebra $A$ has a {\em PBW basis} if the set of standard monomials
 \begin{equation}\label{LinearBasisPBW}
     \{x_1^{a_1}x_2^{a_2}\dotsb x_n^{a_n} \mid a_i\in \mathbb{N} \}
 \end{equation}
 
is a basis for $A$ as a $\Bbbk$-vector space.  This algebra $A$ is called a \textit{PBW algebra}. We know that the Diamond–Composition Lemma provides a method to construct a linear basis for the quotient algebra. The natural question is: under what conditions must a GS basis ensure that the set of irreducible words takes the form described in expression (\ref{LinearBasisPBW})? Our aim is to establish necessary and sufficient conditions for an algebra defined by generators and relations to admit a PBW basis.

Suppose that $A$ is a finitely generated $\Bbbk$-algebra  defined by the following relations
\begin{equation}\label{quadratic relations form 2}
    f_{ji}=x_jx_i-\lambda_{ij}x_ix_j-d_{ij}, \ 1\leq i<j \leq n,
\end{equation}

where $\lambda_{ij}\in \Bbbk$ and $d_{ij}=0$ or $d_{ij} \in \Bbbk\{X\}-{ \rm Span}\{x_jx_i, x_ix_j \}$. Let us ${ \rm Id}(\mathcal{S})$  be the two–sided ideal of $\Bbbk \{ X \}$ generated by $\mathcal{S}=\{f_{ji}\mid 1\leq i<j\leq n \}$. Furthermore, we assume that, with respect to the monomial ordering $\preceq$ in $X^*$, the defining relations satisfy 
\begin{equation*}
    \overline{f}_{ji}=x_jx_i,  \ 1\leq i<j \leq n.
\end{equation*}

We assume that no term of the relations $f_{ji}$ is a subword of the leading word of any element in the set $\mathcal{S}$. According to Li, Polishchuk, and Positselski, an algebra admits a PBW basis (that is, a $\Bbbk$-linear basis) if its $\Bbbk$-linear basis coincides with the set of monomials ${x_1^{a_1}x_2^{a_2}\dotsb x_n^{a_n} \mid a_i\in \mathbb{N}}$. This situation arises precisely when the defining relations of the ideal are of the form (\ref{quadratic relations form 2}).

The following result was proved by Green \cite{Green1999} using Gr\"obner bases and assuming that the elements $d_{ij}$ are homogeneous quadratic polynomials. On the other hand, Li \cite[Theorem 1.5, Ch III]{Li2002} and Levandovskyy \cite[Lemma 2.2, Ch I]{2005Levandovskyy} proved the same result using Gr\"obner bases and choosing $\overline{d}_{ij}\prec x_ix_j$, i.e., the degree of the polynomials $d_{ij}$ is at most 2. Following these ideas, we now present the same result in the framework of GS bases.

\begin{proposition}[{\cite[Theorem 2.14]{Green1999}}]\label{PBWyGSB}
Let $A=\Bbbk\{ X\}/{ \rm Id}(\mathcal{S})$ be a finitely generated $\Bbbk$-algebra. Then, $A$ has a PBW basis if and only if $\mathcal{S}$ is a GS basis for ${ \rm Id}(\mathcal{S})$ with respect to  $\preceq$.
\end{proposition}
\begin{proof}
        If $\mathcal{S}$ is a GS basis for ${ \rm Id}(\mathcal{S})$ with respect to $\preceq$, by Theorem \ref{CDL}, we have that $\mathcal{B}={ \rm Irr}(\mathcal{S})=\{x_{i_1}^{a_1}x_{i_2}^{a_2}\dotsb x_{i_t}^{a_t}\mid i_1<i_2<\dotsb <i_t \}$ is a lineal basis of the algebra $\Bbbk\{ X\}/{ \rm Id}(\mathcal{S})$. In particular, $\mathcal{B}$ is a PBW basis.
        
        Conversely, suppose that each terms of $d_{ij}$ is less than $x_ix_j$ and $\Bbbk\{ X\}/{ \rm Id}(\mathcal{S})$ has a PBW
        basis. Then ${ \rm Irr}(\mathcal{S})=\{x_{i_1}^{a_1}x_{i_2}^{a_2}\dotsb x_{i_t}^{a_t}\mid i_1<i_2<\dotsb <i_t \}$ is a lineal basis of the algebra $\Bbbk\{ X\}/{ \rm Id}(\mathcal{S})$. To show that $\mathcal{S}$ is a GS basis for ${ \rm Id}(\mathcal{S})$ it suffices to show that all compositions are trivial. Let $1\leq i<j<k\leq n$. Then:
        \begin{align*}
        w_{kji}&=\overline{f_{kj}}(a)=(b)\overline{f_{ji}}, \ {\rm with} \ |w_{kji}|<|\overline{f_{kj}}|+|\overline{f_{ji}}|=4 \ {\rm and} \ a,b\in X^*\\
        &=(x_kx_j)( x_i )=( x_k )(x_jx_i)\\
        &=x_kx_jx_i.\\
        (f_{kj},f_{ji})_{w_{kji}}&=(x_kx_j-\lambda_{jk}x_jx_k-d_{jk})x_i-x_k(x_jx_i-\lambda_{ij}x_ix_j-d_{ij})\\
        &=-\lambda_{jk}x_jx_kx_i-d_{jk}x_i+\lambda_{ij}x_kx_ix_j+x_kd_{ij}\\
        &\equiv -\lambda_{jk}x_j(\lambda_{ik}x_ix_k+d_{ik})-d_{jk}x_i+\lambda_{ij}(\lambda_{ik}x_ix_k+d_{ik})x_j+x_kd_{ij}\\
        &=-\lambda_{jk}\lambda_{ik}x_jx_ix_k-\lambda_{jk}x_jd_{ik}-d_{jk}x_i+\lambda_{ij}\lambda_{ik}x_ix_kx_j+\lambda_{ij}d_{ik}x_j\\
        &\ \ +x_kd_{ij}\\
        &\equiv -\lambda_{jk}\lambda_{ik}(\lambda_{ij}x_ix_j+d_{ij})x_k-\lambda_{jk}x_jd_{ik}-d_{jk}x_i\\ 
        &\ \ + \lambda_{ij}\lambda_{ik}x_i(\lambda_{jk}x_jx_k+d_{jk}) +\lambda_{ij}d_{ik}x_j+x_kd_{ij}\\
        &=-\lambda_{jk}\lambda_{ik}\lambda_{ij}x_ix_jx_k-\lambda_{jk}\lambda_{ik}d_{ij}x_k-\lambda_{jk}x_jd_{ik}-d_{jk}x_i\\
        &\ \  +\lambda_{ij}\lambda_{ik}\lambda_{jk}x_ix_jx_k+\lambda_{ij}\lambda_{ik}x_id_{jk}+\lambda_{ij}d_{ik}x_j+x_kd_{ij}\\
        &=-\lambda_{jk}\lambda_{ik}d_{ij}x_k-\lambda_{jk}x_jd_{ik}-d_{jk}x_i+\lambda_{ij}\lambda_{ik}x_id_{jk}+\lambda_{ij}d_{ik}x_j+x_kd_{ij}\\
        &=:r_{ijk}.
        \end{align*}

Suppose that $r_{ijk}$ cannot be further reduced by $\mathcal{S}$. Then $r_{ijk}$ cannot have $x_jx_i$ with $i<j$ as a subword, and so each term of $r_{ijk}$ is in the set ${ \rm Irr}(\mathcal{S})=\{x_{i_1}^{a_1}x_{i_2}^{a_2}\dotsb x_{i_t}^{a_t}\mid i_1<i_2<\dotsb <i_t \}$. As we want the composition $(f_{kj},f_{ji})_{w_{kji}}$ to be trivial, by definition, $r_{ijk}\in { \rm Id}(\mathcal{S})$. By the PBW basis assumption, no element of ${\rm Irr}(\mathcal{S})$ belong to ${ \rm Id}(\mathcal{S})$ other than $0$. Hence $r_{ijk}=0$ in $\Bbbk\{ X\}/{ \rm Id}(\mathcal{S})$ and by Definition \ref{Def.GSB} we have that $\mathcal{S}$ is a  GS basis.
\end{proof}

\subsection{Double extension regular algebras of type (14641)}

We recall the definition of a double extension introduced by Zhang and Zhang \cite{ZhangZhang2008}. Since some typos occurred in their papers \cite[p. 2674]{ZhangZhang2008} and \cite[p. 379]{ZhangZhang2009} concerning the relations that the data of a double extension must satisfy, we follow the corrections presented by Carvalho et al. \cite{Carvalhoetal2011}.

\begin{definition}[{\cite[Definition 1.3]{ZhangZhang2008}; \cite[Definition 1.1]{Carvalhoetal2011}}]\label{DoubleOreDefinition}
Let $R$ be a subalgebra of a $\Bbbk$-algebra $B$.
\begin{itemize}
    \item[\rm (a)] $B$ is called a {\it right double extension} of $R$ if the following conditions hold:
    \begin{itemize}
        \item[\rm (i)] $B$ is generated by $R$ and two new indeterminates $y_1$ and $y_2$;
        \item[\rm (ii)] $y_1$ and $y_2$ satisfy the relation
        \begin{equation}\label{Carvalhoetal2011(1.I)}
        y_2y_1 = p_{12}y_1y_2 + p_{11}y_1^2 + \tau_1y_1 + \tau_2y_2 + \tau_0,
        \end{equation}
        for some $p_{12}, p_{11} \in \Bbbk$ and $\tau_1, \tau_2, \tau_0 \in R$;
        \item[\rm (iii)] $B$ is a free left $R$-module with basis $\left\{y_1^{i}y_2^{j} \mid i, j \ge 0\right\}$.
        \item[\rm (iv)] $y_1R + y_2R + R\subseteq Ry_1 + Ry_2 + R$.
 \end{itemize}
    \item[\rm (b)] A right double extension $B$ of $R$ is called a {\em double extension} of $R$ if
    \begin{enumerate}
        \item [\rm (i)] $p_{12} \neq 0$;
        \item [\rm (ii)] $B$ is a free right $R$-module with basis $\left\{ y_2^{i}y_1^{j}\mid i, j \ge 0\right\}$;
        \item [\rm (iii)] $y_1R + y_2R + R = Ry_1 + Ry_2 + R$.
    \end{enumerate}
\end{itemize}
\end{definition}

Condition (a)(iv) from Definition \ref{DoubleOreDefinition} is equivalent to the existence of two maps
\[
\sigma = \begin{bmatrix}
    \sigma_{11} & \sigma_{12} \\ \sigma_{21} & \sigma_{22}
\end{bmatrix}: R\to M_{2\times 2}(R)\quad {\rm and}\quad \delta = \begin{bmatrix} \delta_1 \\ \delta_2  \end{bmatrix}: R\to M_{2\times 1}(R),
\]

such that
\begin{equation}\label{Carvalhoetal2011(1.II)}
    \begin{bmatrix}
        y_1 \\ y_2
    \end{bmatrix} r = \sigma(r) \begin{bmatrix}
        y_1 \\ y_2
    \end{bmatrix} + \delta(r) \quad {\rm for\ all}\ r\in R.
\end{equation}

If $B$ is a right double extension of $R$, we write $B = R_P[y_1, y_2;\sigma, \delta, \tau]$, where $P = (p_{12}, p_{11})$ with elements belonging to $\Bbbk$, $\tau = \{\tau_0, \tau_1, \tau_2\} \subseteq R$, and $\sigma, \delta$ are as above. $P$ is called a {\em parameter} and $\tau$ a {\em tail}, while the set $\{P, \sigma, \delta, \tau\}$ is said to be the {\em DE-data}. One of the particular cases of the double extensions is presented by Zhang and Zhang \cite[Convention 1.6.(c)]{ZhangZhang2008} as a {\it trimmed double extension}, for which $\delta$ is the zero map and $\tau = \{0, 0, 0\}$. We use the short notation $R_p[y_1, y_2; \sigma]$ to denote this subclass of extensions. 

\begin{example}[{\cite[Example 4.1]{ZhangZhang2008}}]\label{ZhangZhang2008Example4.1}
Let $R = \Bbbk[x]$. As expected, there are so many different right double extensions of $R$ and if we assume that ${\rm deg}\ x = {\rm deg}\ y_1 = {\rm deg}\ y_2 = 1$, then all connected graded double extensions $\Bbbk[x]_P[y_1, y_2; \sigma, \delta, \tau]$ are regular algebras of global dimension three investigated by Artin and Schelter \cite{ArtinSchelter1987}. 
\end{example}

Zhang and Zhang \cite{ZhangZhang2009} were interested only in regular algebras $B$ of dimension four that are generated in degree one, and in the case that $B$ is generated by four elements, the projective resolution of the trivial module $\Bbbk_B$ is of the form (\ref{LISTResolution}), that is, algebras {\em of type} (14641). The next proposition shows explicitly the relation of these algebras with Ore extensions.
\begin{equation}
0 \xrightarrow{} B(-4) \xrightarrow{} B(-3)^{\oplus 4} \xrightarrow{} B(-2)^{\oplus 6} \to B(-1)^{\oplus 4} \to B \xrightarrow{} \Bbbk_B \xrightarrow{} 0.
\end{equation}

\begin{proposition}[{\cite[Theorem 0.1]{ZhangZhang2009}}]
Let $B$ be a connected graded algebra generated by four elements of degree one. If $B$ is a double extension $R_P[y_1, y_2; \sigma, \tau]$ where $R$ is an Artin-Schelter regular algebra of dimension two, then the following assertions hold:
\begin{enumerate}
    \item [\rm (1)] $B$ is a strongly Noetherian, Auslander regular and Cohen-Macaulay domain.

    \item [\rm (2)] $B$ is of type {\rm (}14641{\rm )}. As a consequence, $B$ is Koszul.

    \item [\rm (3)] If $B$ is not isomorphic to an Ore extension of an Artin-Schelter regular algebra of dimension three, then the trimmed double extension $R_P[y_1, y_2; \sigma]$ is isomorphic to one of 26 families.
\end{enumerate}
\end{proposition}

In Tables \ref{tab:1-DOE}, \ref{tab:2-DOE},  \ref{tab:3-DOE} and  \ref{tab:4-DOE} we present the detailed list of the 26 families following the labels $\mathbb{A}, \mathbb{B}, \dotsc, \mathbb{Z}$ used in \cite{ZhangZhang2009}.

\begin{example}
    Since the base field $\Bbbk$ is algebraically closed, $R$ is isomorphic to $\Bbbk_q[x_1, x_2] = \mathcal{O}_q(\Bbbk)$ with the relation $x_2 x_1 = qx_1 x_2$ (note that $Q = (q, 0)$, the {\em Manin's plane}), or $\Bbbk_J[x_1, x_2] = \mathcal{J}(\Bbbk)$ with the relation $x_2x_1 = x_1x_2 + x_1^2$ (here, $Q = J = (1,1)$, the {\em Jordan's plane}) \cite[Theorem 1.4]{Shirikov2005} or \cite[Lemma 2.4]{ZhangZhang2009}. Manin's plane and Jordan's plane are the only regular algebras of global dimension two \cite[Examples 1.8 and 1.9]{Rogalski2023}.

In general, we will write $\Bbbk_Q[x_1, x_2]$, with $Q = (q_{12}, q_{11})$, and 
\[
\Bbbk_Q[x_1, x_2] = \Bbbk\{x_1, x_2\} / \langle x_2x_1 - q_{11}x_1^2 - q_{12}x_1x_2\rangle,
\]

and for the computation we set $Q$ to be either $(1, 1)$ or $(q, 0)$.
\end{example}

\section{reduction of compositions using \texttt{MATLAB}} \label{Matlab}

We illustrate the use of a code implemented in \texttt{MATLAB} developed by {\em Santiago Cuartas}\footnote{Email: \texttt{dcuartas@unal.edu.co}} to calculate and perform the composition reduction steps of the Shirshov's algorithm. It is important to note that this algorithm only performs operations on polynomials and executes the reduction procedure; the remaining steps of Shirshov’s algorithm must still be carried out manually. For instance, the computation of ambiguities and the selection of leading terms require manual intervention, since our implementation does not incorporate an ordering procedure. The functions required for the coding of this algorithm are provided in \cite{Herrera2024}.
\medskip

We will illustrate the implementation of the code using the GS basis computed in Example \ref{ExampleSklyanin}. The idea is to define $x,y,z$ by using a file in \texttt{MATLAB} or a \texttt{Script} of type \texttt{struct}, where the first, second and third fields indicate the number of monomials, the coefficients of each monomial and the monomials, respectively. The indeterminates $x, y, z $ are denoted as $\texttt{1,2,3}$ respectively. This is done as follows:

\begin{Verbatim}[commandchars=\\\{\}]
    P1.a=1;
    P1.b=\{1\};
    P1.c=\{\textcolor{violet}{'1'}\}; 

    P2.a=1;
    P2.b=\{1\};
    P2.c=\{\textcolor{violet}{'2'}\}; 

    P3.a=1;
    P3.b=\{1\};
    P3.c=\{\textcolor{violet}{'3'}\}; 
\end{Verbatim}

In the case of a $n$ generator, it will be defined in the same way. Now, the relations defined above depend on two parameters, $a$ and $s$, so when applying the reduction algorithm we could say that they behave as quotients of polynomial rings in two indeterminates. 
Therefore, in \texttt{Script} these parameters $a$ and $s$ are created as variables of type \texttt{sym} as follows: 
\begin{Verbatim}[commandchars=\\\{\}]
    format \textcolor{violet}{rational}
    a = sym(\textcolor{violet}{'a'});
    s = sym(\textcolor{violet}{'s'});
\end{Verbatim}
The condition $a^3=-1$ will be written later. Since the monomial order in $X^*$ was not encoded, it is important to note that the transformation of the relations to monic polynomials must be done by hand. This set $S$ consists of the monic polynomials $f_1=x^2-\frac{1}{s}yz+\frac{a}{s}zy$, $f_2=xy-ayx-sz^2$, $f_3=xz-\frac{1}{a}zx+\frac{s}{a}y^2$. In \texttt{Script}, the above polynomials are written similarly to the generators as data of type \texttt{struc} as follows:
\begin{Verbatim}[commandchars=\\\{\}]
    F1.a=3;
    F1.b=\{1,-1/s,a/s\};
    F1.c=\{\textcolor{violet}{'11'},\textcolor{violet}{'23'},\textcolor{violet}{'32'}\};

    F2.a=3;
    F2.b=\{1,-a,-s\};
    F2.c=\{\textcolor{violet}{'12'},\textcolor{violet}{'21'},\textcolor{violet}{'33'}\};

    F3.a=3;
    F3.b=\{1,-1/a,s/a\};
    F3.c=\{\textcolor{violet}{'13'},\textcolor{violet}{'31'},\textcolor{violet}{'22'}\};
\end{Verbatim}

It is important to highlight that in the first position of the field \texttt{Fi.c} there will always be the principal monomial of each monic polynomial $f_i$, for $i = 1, 2, 3$. The next step is to know what the possible compositions of the set $S$ are, that is, to calculate the possible ambiguities between $\overline{f_1}, \overline{f_2}$, and $\overline{f_3}$. This task must be done by hand. As seen in Example \ref{ExampleSklyanin}, the first composition is $(f_1,f_2)_{w_1}=f_1(y)-(x)f_2$. It is clear that its computation requires a reduction system that consists of the polynomials of the set $S$, which will be modified every time. The way to write this composition is as follows:

\begin{Verbatim}[commandchars=\\\{\}]
    K1=multpol(F1,P2)  
    K2=multpol(P1,F2) 
    P=Sumpol(K1,K2,1)
    X=\{F1,F2,F3\};
    Y=\{\{a^3,-1\}\};
    r=0
    r1=1;
    F=\{ \}
    \textcolor{blue}{while} r==0
        disp([\textcolor{violet}{'Make #'},num2str(r1),\textcolor{violet}{' reduction'}])
        r1=r1+1;
        disp([\textcolor{violet}{'Polynomial to reduce'}])
        P
        P11=Redpol(P,X);
        disp([\textcolor{violet}{'reduced polynomial 1'}])
        P11
        P1=Redpol2(P11,Y);
        disp([\textcolor{violet}{'reduced polynomial 2'}])
        P1
        F\{r1\}=P1
        \textcolor{blue}{if} 0==IgualPol(P,P1)
            P=P1;
        \textcolor{blue}{else}
            r=1
        \textcolor{blue}{end}
    \textcolor{blue}{end}
    
    >> P1= struct with fields:
            a: 2
            b:\{-(s^3 + 1)/(a^2*s)  -(s^3 + 1)/(a*s)\}
            c:\{'322'  '223'\} 
\end{Verbatim}

where \texttt{K1} is the product of $f_1$ and $y$,  \texttt{K2} is the product of $x$ and $f_2$, and \texttt{P} is the difference of both products, \texttt{X} is the reduction system, and \texttt{Y} is the way the condition is written $a^3=-1$. The functions \texttt{multpol}, \texttt{Sumpol} and other are defined in \cite{Herrera2024}. The result returned by the code is \texttt{P1}, i.e., $(f_1,f_2)_{w_1}\equiv-\left( \frac{1+s^3}{sa^2}\right)zy^2-\left( \frac{1+s^3}{sa}\right)y^2z$. According to the established monomial order, transforming this polynomial into a monic one and simplifying in the same way, the new polynomial $f_4:=y^2z+\frac{1}{a}zy^2$ is created as
\begin{Verbatim}[commandchars=\\\{\}]
    F4.a=2;
    F4.b=\{1,1/a\};
    F4.c=\{\textcolor{violet}{'223'},\textcolor{violet}{'322'}\}; 
\end{Verbatim}
The next composition to calculate is between $f_1$ and $f_3$, whence the only part that is modified in the code is \texttt{K1}, \texttt{K2} and \texttt{X}. This is done as follows:
\begin{Verbatim}[commandchars=\\\{\}]
    K1=multpol(F1,P3)  
    K2=multpol(P1,F3) 
    P=Sumpol(K1,K2,1)
    X=\{F1,F2,F3,F4\};

    >> P1= struct with fields:
            a: 2
            b:\{-(s^3 + 1)/s  -(s^3 + 1)/(a*s)\}
            c:\{'233'  '332'\}
\end{Verbatim}

The result returned by the code is \texttt{P1}, that is, $(f_1,f_3)_{w_3}=-\left(\frac{s^3+1}{s} \right)yz^2-\left(\frac{s^3+1}{as} \right)z^2y$. Transforming this polynomial into a monic one and simplifying in the same way, the new polynomial $f_5:=yz^2+\frac{1}{a}z^2y$ is created as
\begin{Verbatim}[commandchars=\\\{\}]
    F5.a=2;
    F5.b=\{1,1/a\};
    F5.c=\{\textcolor{violet}{'233'},\textcolor{violet}{'332'}\}; 
\end{Verbatim}

Finally, we calculate the composition between $f_1$ and $f_1$:
\begin{Verbatim}[commandchars=\\\{\}]
    K1=multpol(F1,P1)  
    K2=multpol(P1,F1) 
    P=Sumpol(K1,K2,1)
    X=\{F1,F2,F3,F4,F5\};

    >> P1= struct with fields:
            a: 0 
            b: \{ \}
            c: \{ \}
\end{Verbatim}
Just as we saw in Example \ref{ExampleSklyanin}, the other compositions cannot be calculated. Now, we will perform the second iteration of Shirshov's algorithm and verify that the set $S_1:=S\cup \{f_4,f_5 \}$ is a GS basis. As we saw, there are only three possible compositions and they are also trivial.
\begin{Verbatim}[commandchars=\\\{\}]
    K1=multpol(F2,multpol(P2,P3))  
    K2=multpol(P1,F4) 
    P=Sumpol(K1,K2,1)
    X=\{F1,F2,F3,F4,F5\};

    >> P1= struct with fields:
            a: 0 
            b: \{ \}
            c: \{ \}
\end{Verbatim}
\begin{Verbatim}[commandchars=\\\{\}]
    K1=multpol(F2,multpol(P3,P3))  
    K2=multpol(P1,F5) 
    P=Sumpol(K1,K2,1)
    X=\{F1,F2,F3,F4,F5\};

    >> P1= struct with fields:
            a: 0 
            b: \{ \}
            c: \{ \}
\end{Verbatim}
\begin{Verbatim}[commandchars=\\\{\}]
    K1=multpol(F4,P3)  
    K2=multpol(P2,F5) 
    P=Sumpol(K1,K2,1)
    X=\{F1,F2,F3,F4,F5\};

    >> P1= struct with fields:
            a: 0 
            b: \{ \}
            c: \{ \}
\end{Verbatim}

Therefore, as verified in Example \ref{ExampleSklyanin} the set $S^c$ is a GS basis for the ideal ${\rm Id}(S)$.

\section{Presenting Double Extension Regular Algebras via Gr\"obner-Shirshov Bases}\label{GS bases for 14641}

Note that each family of double extensions in Tables \ref{tab:1-DOE}, \ref{tab:2-DOE}, \ref{tab:3-DOE} and \ref{tab:4-DOE} can be expressed as a quotient of the free algebra $\Bbbk \{ y_1, y_2, x_1, x_2 \}$ by the ideal of relations ${\rm Id}(S)$. To calculate the GS bases of these families, the following procedure is followed. Consider the alphabet \( X = \{y_1, y_2, x_1, x_2\} \) equipped with the \texttt{deglex} order on the monoid \( X^* \), where \( y_1 \prec y_2 \prec x_1 \prec x_2 \).Then, we select the relations defining the double extension as polynomials, arranging them so that the leading coefficients of the main words are \( 1 \). Taking the set \( S \) to consist of all monic polynomials, we compute all possible compositions (see Definition~\ref{Def.Composition}) and, in accordance with Definition~\ref{Def.GSB}, we determine whether \( S \) forms a GS basis. For the following algebras, it was not necessary to apply Shirshov's algorithm, as we were able to find the GS bases using only the definition.

\subsection{Case I} Let us consider the relations of the double extension $\mathbb{A}$
\begin{align*}
    x_2x_1 &\ = x_1 x_2, &  y_1x_1 &\ = x_1y_1,  & y_2x_1 &\ = x_1y_2, \\
    y_2y_1 &\ = y_1y_2+y_{1}^{2}, &  y_1x_2&\ =x_2y_1+x_1y_2, & y_2x_2&\ =-2x_2y_1-x_1y_2+x_2y_2.
\end{align*}

By taking the principal leadings of the relations with respect to the monomial order defined above, we have the polynomials 
\begin{align*}
     f_1&\ = x_2x_1- x_1 x_2, & f_4&\ = y_2y_1-y_1y_2-y_{1}^{2},\\
     f_2&\ =x_1y_1-y_1x_1, & f_5&\ =x_2y_1-y_1x_2+x_1y_2,\\
     f_3&\ =x_1y_2-y_2x_1, & f_6&\ = x_2y_2-y_2x_2-2x_2y_1-x_1y_2.
 \end{align*}
 
Note that $f_5$ and $f_6$ can be reduced as
\begin{align*}
    f_5& =x_2y_1-y_1x_2+x_1y_2\equiv x_2y_1-y_1x_2+y_2x_1,\\
    f_6& =x_2y_2-y_2x_2-2x_2y_1-x_1y_2\\
       &\equiv x_2y_2-y_2x_2-2(y_1x_2-y_2x_1)-y_2x_1\\
       &= x_2y_2-y_2x_2-2y_1x_2+y_2x_1
\end{align*}
 
Thus, we obtain that the set $S$ is formed by the following monic polynomials:
\begin{equation*}
   S=  \begin{Bmatrix}
x_2x_1- x_1 x_2 &  x_1y_1-y_1x_1, & x_1y_2-y_2x_1 , \\
y_2y_1-y_1y_2-y_{1}^{2}, &  x_2y_1-y_1x_2+y_2x_1, & x_2y_2-y_2x_2-2y_1x_2+y_2x_1
\end{Bmatrix}
\end{equation*}

Next, we proceed to calculate all the possible ambiguities and their respective compositions in case they exist. Just as in the previous section, the only compositions that can be formed are such that the leading terms of the polynomials end and begin with the same letter.  Using the code, it is easy to see that the compositions $ (f_1,f_2)_{w_1}$, $(f_1,f_3)_{w_2}$, $ (f_3,f_4)_{w_3}$ and $  (f_6,f_4)_{w_4}$ are the only  that exist, and they are trivial.
Hence, the set $S$ is formed by monic polynomials is a GS basis for the ideal ${ \rm Id}(S)$.

\subsection{Case II}
For the double extension $\mathbb{B}$, we have the relations:
\begin{align*}
    x_2x_1 &\ = px_1 x_2, &  y_2y_1 &\ = py_1y_2,  & y_1x_1 &\ = x_2y_2, \\
    y_1x_2&\ =x_1y_2,  & y_2x_1 &\ = -x_2y_1, &  y_2x_2&\ =x_1y_1,
\end{align*}

where $p^2 = -1$. Following the same procedure as above, but with these new relationships, we obtain that the set $S$ is formed by the following monic polynomials:
\begin{align*}  
f_1&\ = x_2x_1- px_1 x_2, &  f_2&\ = y_2y_1-py_1y_2, & f_3&\ = x_2y_2-y_1x_1 , \\
f_4&\ =x_1y_2-y_1x_2, &  f_5&\ =x_2y_1+y_2x_1, & f_6&\ = x_1y_1-y_2x_2
\end{align*}
 
 Next, we proceed to calculate all the possible ambiguities and their respective compositions in case they exist. So the only possible compositions $(f_1,f_4)_{w_1}$, $(f_1,f_6)_{w_2}$, $(f_3,f_2)_{w_3}$ and $ (f_4,f_2)_{w_4}$ are trivial. Hence, the set $S$ is formed by monic polynomials is a GS basis for the ideal ${ \rm Id}(S)$. 

\subsection{Case III} 

The relations of the double extension $\mathbb{D}$ are given by:
\begin{align*}
    x_2x_1 &\ = -x_1 x_2, & y_2y_1 &\ = py_1y_2,  & y_1x_1 &\ = -px_1y_1, \\
   y_1x_2&\ =-p^2x_2y_1+x_1y_2,  & y_2x_1 &\ = px_1y_2, & y_2x_2&\ =x_1y_1+x_2y_2,
\end{align*}

where $p\in \{-1,1\}$. We obtain that the set $S$ is formed by the following monic polynomials:
\begin{align*} 
f_1&\ = x_2x_1+x_1 x_2 & f_2&\ = y_2y-py_1y_2, & f_3&\ = x_1y_1+py_1x_1 , \\
f_4&\ =x_2y_1+y_1x_2-\frac{1}{p}y_2x_1, &  f_5&\ =x_1y_2-\frac{1}{p}y_2x_1, & f_6&\ = x_2y_2-y_2x_2+y_1x_1.
\end{align*}
Note that $f_4$ can be reduced by $f_5$. When $p=1$ or $p=-1$, the set only differs in a some signs, and this does not alter the choice of the leading term for each value of $p$. Therefore, in any case, the compositions that can be formed are $(f_1,f_3)_{w_1}$, $(f_1,f_5)_{w_2}$, $(f_5,f_2)_{w_3}$ and $(f_6,f_2)_{w_4}$ and are also trivial. Therefore, $S$ is a GS basis for the algebra $\mathbb{D}$ when $p\in \{-1,1\}$. 

\subsection{Case IV}
We establish the relations for the double extension  $\mathbb{E}$: 
\begin{align*}
    x_2x_1 &\ = -x_1 x_2, & y_2y_1 &\  = py_1y_2, & y_1x_1&\ = x_1y_2+x_2y_2, \\
   y_1x_2&\  =x_1y_2-x_2y_2, & y_2x_1 &\ = -x_1y_1+x_2y_1, & y_2x_2&\  =x_1y_1+x_2y_1,
\end{align*}
where $p^2=-1$. Similarly, based on these relationships, we transform them into monic polynomials as follows:
\begin{align*}
    f_1&\ =x_2x_1 +x_1 x_2, & f_2&\ =y_2y_1 -py_1y_2,  & f_3&\ = x_2y_2+x_1y_2-y_1x_1,\\
   f_4&\ =x_2y_2-x_1y_2+y_1x_2,  & f_5&\ =x_2y_1-x_1y_1-y_2x_1, & f_6&\ =x_2y_1+x_1y_1-y_2x_2.
\end{align*}
Note that $f_4$ can be reduced by $f_4$, that is
\begin{align*}
    f_4&=x_2y_2-x_1y_2+y_1x_2\\
    &\equiv (-x_1y_2+y_1x_1)-x_1y_2+y_1x_2\\
    &=-2x_1y_2+y_1x_1+y_1x_2\\
    &=x_1y_2-\frac{1}{2}y_1x_1-\frac{1}{2}y_1x_2.
\end{align*}
Now, $f_3$ can be reduced by the new $f_4$ as follows:
\begin{align*}
    f_3&= x_2y_2+x_1y_2-y_1x_1\\
    &\equiv x_2y_2+(\frac{1}{2}y_1x_1+\frac{1}{2}y_1x_2)-y_1x_1\\
    &=x_2y_2-\frac{1}{2}y_1x_1+\frac{1}{2}y_1x_2.
\end{align*}
Similarly, $f_6$ can be reduced using $f_5$
\begin{align*}
    f_6&=x_2y_1+x_1y_1-y_2x_2\\
    &\equiv (x_1y_1+y_2x_1)+x_1y_1-y_2x_2\\
    &=2x_1y_1+y_2x_1-y_2x_2\\
    &=x_1y_1-\frac{1}{2}y_2x_1-\frac{1}{2}y_2x_2,
\end{align*}
and now, $f_5$ can be reduced by means of the new $f_6$
\begin{align*}
    f_5&=x_2y_1-x_1y_1-y_2x_1\\
    &\equiv x_2y_1+(\frac{1}{2}y_2x_1-\frac{1}{2}y_2x_2)-y_2x_1\\
    &=x_2y_1-\frac{1}{2}y_2x_1-\frac{1}{2}y_2x_2.
\end{align*}
Similarly, we have that the following compositions $(f_1,f_4)_{w_1}$, $(f_1,f_6)_{w_2}$, $(f_3,f_2)_{w_3}$ and $(f_4,f_2)_{w_4}$ are the only ones that exist and are trivial. Therefore, the set
\begin{equation*}
   S=  \begin{Bmatrix}
 x_2x_1 +x_1 x_2, &  y_2y_1 -py_1y_2, \\
 x_2y_2-\frac{1}{2}y_1x_1+\frac{1}{2}y_1x_2, &
x_1y_2-\frac{1}{2}y_1x_1-\frac{1}{2}y_1x_2,\\
x_2y_1-\frac{1}{2}y_2x_1-\frac{1}{2}y_2x_2, &  x_1y_1-\frac{1}{2}y_2x_1-\frac{1}{2}y_2x_2
\end{Bmatrix},
\end{equation*}
 is a GS basis for the ideal ${ \rm Id}(S)$. 

\subsection{Case V}
Let us consider the relations of the double extension $\mathbb{G}$ for $p\not = 0, \pm 1$ and $\alpha \neq 0$:
\begin{align*}
    x_2x_1 &\ = x_1 x_2, & y_2y_1 &\ = py_1y_2, \\
    y_1x_1&\ = px_1y_1, &   y_2x_1&\ = px_1y_2, \\
   y_1x_2&\ =px_1y_1+p^2x_2y_1+x_1y_2, & y_2x_2&\ =\alpha x_1y_1-x_1y_2+x_2y_2.
\end{align*}

Therefore, $S$ is the set of the following monic polynomials:
\begin{align*}
    f_1&\ =x_2x_1 -x_1 x_2, & f_2&\ =y_2y_1 -py_1y_2,  \\
    f_3&\ = x_1y_1-\frac{1}{p}y_1x_1 , & f_4&\ =x_1y_2-\frac{1}{p}y_2x_1,\\
    f_5&\ = x_2y_1-\frac{1}{p^2}y_1x_2+\frac{1}{p^2}y_1x_1+\frac{1}{p^3}y_2x_1, & f_6&\ = x_2y_2+\frac{\alpha}{p} x_1y_1-\frac{1}{p}y_2x_1-y_2x_2.
\end{align*}

Thus, the possible compositions are $(f_1,f_4)_{w_1}$, $(f_1,f_3)_{w_2}$, $(f_4,f_2)_{w_3}$ and $(f_6,f_2)_{w_4}$ and, moreover, they are trivial. Thus, the set $S$ defined by the above polynomials is a GS basis the ideal ${ \rm Id}(S)$.

\subsection{Case VI} Let us consider the relations of the double extension $\mathbb{K}$
\begin{align*}
   x_2x_1 &\ = q x_1 x_2, &  y_2y_1&\ = -y_1y_2, &\ y_1x_1&\ = x_1y_1,\\
  y_1x_2&\ =x_2y_2, &  y_2x_1&\ = x_1y_2, &\ y_2x_2&\ =\alpha x_2y_1, 
\end{align*}
where $q\in \{-1,1\}$ and $\alpha \not = 0$. Thus, the set $S$ is formed by the following monic polynomials:
\begin{align*}
    f_1 &\ =x_2x_1- q x_1 x_2, & f_2&\ = y_2y_1 +y_1y_2, & f_3 &\ = x_1y_1-y_1x_1,\\
    f_4&\ =x_2y_2-y_1x_2, & f_5&\ =x_1y_2-y_2x_1, & f_6&\ = x_2y_1-\frac{1}{\alpha}y_2x_2.
\end{align*}
Thus, the possible compositions in the set $S$ are $(f_1,f_3)_{w_1}$, $(f_1,f_5)_{w_2}$, $(f_4,f_2)_{w_3}$, and $(f_5,f_2)_{w_4}$ and are trivial. Thus, the set $S$ defined by the above polynomials is  a GS basis  for the ideal ${ \rm Id}(S)$.

\subsection{Case VII} 
For the double extension $\mathbb{L}$, we have the relations:
\begin{align*}
    x_2x_1 &\ = qx_1 x_2, & y_2y_1 &\ = -y_1y_2, & y_1x_1 &\ = \alpha x_1y_2,\\
     y_1x_2&\ =x_2y_2, & y_2x_1 &\ = \alpha x_1y_1, &\ y_2x_2&\ =x_2y_1,
\end{align*}

where $q\in \{-1,1\}$ and $\alpha \not = 0$.  Thus, the set $S$ is formed by the following monic polynomials:
\begin{align*}
    f_1&\ =x_2x_1-qx_1 x_2, & f_2&\ =y_2y_1 +y_1y_2,& f_3&\ =x_1y_2-\frac{1}{\alpha}y_1x_1, \\
    f_4&\ =x_2y_2-y_1x_2, & f_5&\ = x_1y_1-\frac{1}{\alpha}y_2x_1, & f_6&\ =x_2y_1-y_2x_2.
\end{align*}

Thus, the possible compositions in the set $S$ are $(f_1,f_3)_{w_1}$, $(f_1,f_5)_{w_2}$, $(f_3,f_2)_{w_3}$, and $(f_4,f_2)_{w_4}$ and are trivial. Thus, the set $S$ defined by the above polynomials is a basis for GS for the ideal ${ \rm Id}(S)$.

\subsection{Case VIII} 
The double extension $\mathbb{Q}$ satisfies the following relations:
\begin{align*}
    x_2x_1 &= -x_1 x_2, &\ y_2y_1 &= -y_1y_2, &\ y_1x_1 &= x_1y_2, \\
    y_2x_1 &= -x_1y_1, &\ y_1x_2&=x_1y_1+x_2y_1+x_1y_2, &\ y_2x_2&=x_1y_1-x_1y_2+x_2y_2.
\end{align*}
Therefore, the set $S$ consists of the following monic polynomials:
\begin{align*}
    f_1&=x_2x_1 +x_1 x_2, &\ f_2&=y_2y_1+y_1y_2, \\
    f_3&=x_1y_2+y_1x_2, &\ f_4&= x_2y_1+y_2x_1,\\
    f_5&= x_2y_1+y_1x_1-y_2x_1-y_1x_2, &\ f_6&= x_2y_2-y_2x_1-y_1x_1-y_2x_2.
\end{align*}
Thus, the possible compositions are $(f_1,f_4)_{w_1}$, $(f_1,f_3)_{w_2}$, $(f_3,f_2)_{w_3}$ and $(f_6,f_2)_{w_4}$ and, moreover, they are trivial. Thus, the set $S$ defined by the above polynomials is a GS basis the ideal ${ \rm Id}(S)$.

\subsection{Case IX}
The defining relations for the double extension $\mathbb{R}$ are:
\begin{align*}
    x_2x_1 &= -x_1 x_2, &\ y_2y_1 &= -y_1y_2, &\ y_1x_2&=x_1y_2, \\
    y_2x_1 &= x_2y_1, &\ 
    y_1x_1 &= x_1y_1+x_2y_1+x_1y_2, &\ 
    y_2x_2&=-x_2y_1-x_1y_2+x_2y_2.
\end{align*}
It follows that the following monic polynomials form the set $S$:
\begin{align*}
    f_1&=x_2x_1 +x_1 x_2, &\ f_2&=y_2y_1+y_1y_2, \\
    f_3&= x_1y_2-y_1x_2, &\ f_4&=x_2y_1-y_2x_1,\\
    f_5&=x_1y_1+y_2x_1+y_1x_2-y_1x_1, &\ f_6&=x_2y_2-y_2x_1-y_1x_2-y_2x_2.
\end{align*}
Thus, the possible compositions are $(f_1,f_5)_{w_1}$, $(f_1,f_3)_{w_2}$, $(f_3,f_2)_{w_3}$ and $(f_6,f_2)_{w_4}$ and, moreover, they are trivial. Thus, the set $S$ defined by the above polynomials is a GS basis the ideal ${ \rm Id}(S)$.

\subsection{Case X}
For the double extension $\mathbb{V}$, we have the relations:
\begin{align*}
    x_2x_1 &= x_1 x_2, &\ y_2y_1 &= -y_1y_2, &\ y_1x_2&= x_2y_1, \\
    y_2x_2&=x_2y_2, &\ 
    y_1x_1 &= x_2y_1+x_1y_2, &\ 
    y_2x_1 &= -x_1y_1+x_2y_1.
\end{align*}
Therefore, the set $S$ consists of the following monic polynomials:
\begin{align*}
    f_1&=x_2x_1 -x_1 x_2, &\ f_2&=y_2y_1+y_1y_2, &\ f_3&= x_2y_1-y_1x_2, \\
    f_4&=x_2y_2-y_2x_2, &\ f_5&=x_1y_2+y_1x_2-y_1x_1, &\ f_6&=x_1y_1+y_2x_1-y_1x_2.    
\end{align*}
Thus, the possible compositions are $(f_1,f_6)_{w_1}$, $(f_1,f_5)_{w_2}$, $(f_4,f_2)_{w_3}$ and $(f_5,f_2)_{w_4}$ and, moreover, they are trivial. Thus, the set $S$ defined by the above polynomials is a GS basis the ideal ${ \rm Id}(S)$. 

\subsection{Case XI}
The following relations hold in the double extension $\mathbb{X}$:
\begin{align*}
    x_2x_1 &= x_1 x_2, &\ y_2y_1 &= -y_1y_2, &\ y_1x_1 &= x_1y_2, \\
    y_2x_1 &= x_1y_1, &\ 
    y_1x_2&= x_1y_2+x_2y_2, &\ 
    y_2x_2&=x_1y_1+x_2y_1.
\end{align*}
Therefore, we define the set $S$ as the collection of these monic polynomials:
\begin{align*}
    f_1&=x_2x_1 -x_1 x_2, &\ f_2&=y_2y_1+y_1y_2, &\ f_3&= x_1y_2-y_1x_1, \\
    f_4&= x_1y_1-y_2x_1, &\ f_5&=x_2y_2+y_1x_1-y_1x_2, &\ f_6&=x_2y_1+y_2x_1-y_2x_2.
\end{align*}
Thus, the possible compositions are $(f_1,f_3)_{w_1}$, $(f_1,f_4)_{w_2}$, $(f_3,f_2)_{w_3}$ and $(f_5,f_2)_{w_4}$ and, moreover, they are trivial. Thus, the set $S$ defined by the above polynomials is a GS basis the ideal ${ \rm Id}(S)$. 

\subsection{Case XII}
For the double extension $\mathbb{Y}$, we have the relations:
\begin{align*}
    x_2x_1 &= x_1 x_2, &\ y_2y_1 &= -y_1y_2, &\ y_1x_1 &= x_1y_1, \\
    y_2x_1 &= x_1y_2, &\ 
    y_1x_2&= \alpha x_1y_1-x_2y_1+x_1y_2, &\ 
    y_2x_2&=x_1y_1+\alpha x_1y_2-x_2y_2,
\end{align*}
where $\alpha \in \Bbbk$. It follows that the following monic polynomials form the set $S$:
\begin{align*}
    f_1&=x_2x_1 -x_1 x_2, &\ f_2&=y_2y_1+y_1y_2, \\
    f_3&=x_1y_1-y_1x_1 &\ f_4&= x_1y_2-y_2x_1\\
    f_5&=x_2y_1-y_2x_1-\alpha y_1x_1+y_1x_2 &\ f_6&=x_2y_2-y_1x_1-\alpha y_2x_1+y_2x_2
\end{align*}
Thus, the possible compositions are $(f_1,f_3)_{w_1}$, $(f_1,f_4)_{w_2}$, $(f_4,f_2)_{w_3}$ and $(f_6,f_2)_{w_4}$ and, moreover, they are trivial. Thus, the set $S$ defined by the above polynomials is a GS basis the ideal ${ \rm Id}(S)$. 

\begin{remark} For all the algebras studied previously, we have that the set ${ \rm Irr}(S)=\{y_1^{i_1}y_2^{i_2}x_1^{i_3}x_2^{i_4} \mid i_1,i_2,i_3,i_4\geq 0 \}$ forms a linear basis, by the Composition-Diamond Lemma for free algebras \ref{CDL}. Furthermore, by the Proposition \ref{PBWyGSB}, the set ${ \rm Irr}(S)$ is a PBW basis for these algebras. For the remaining 14 algebras, it is necessary to perform more than one iteration of Shirshov's algorithm, and in some cases this can be tedious due to the form of the polynomial coefficients.
\end{remark}

\begin{landscape}
\begin{table}[h]
\begin{center}
\begin{tabular}{|c|c|c|c|}
\hline
Double extension & Relations defining the double extension & Conditions & GS bases finite \\ 
\hline
\multirow{3}{*}{$\mathbb{A}$} & $x_2x_1 = x_1 x_2,\quad y_2y_1 = y_1y_2+y_{1}^{2}$, & &  \\
             &$y_1x_1 = x_1y_1,  \quad y_1x_2=x_2y_1+x_1y_2$, &  & \checkmark \\
             & $y_2x_1 = x_1y_2, \quad y_2x_2=-2x_2y_1-x_1y_2+x_2y_2$.  & & \\
             \hline
\multirow{3}{*}{$\mathbb{B}$} &  $x_2x_1 = px_1 x_2, \quad y_2y_1 = py_1y_2$, & &  \\
             &  $y_1x_1 = x_2y_2,  \quad y_1x_2=x_1y_2$,  & $p^2 = -1$ &\checkmark \\
             &  $y_2x_1 = -x_2y_1, \quad y_2x_2=x_1y_1$. & & \\ \hline
\multirow{5}{*}{$\mathbb{C}$} & $x_2x_1 = px_1 x_2, \quad y_2y_1 = py_1y_2$, & & \\
             &  $y_1x_1 = -x_1y_1+p^2x_2y_1+x_1y_2-px_2y_2$, & & \\
             &  $y_1x_2 =-px_1y_1+x_2y_1+x_1y_2-px_2y_2$, & $p^2+p+1 = 0$ &  \\
             &  $y_2x_1 = -px_1y_1-2p^2x_2y_1+px_1y_2-px_2y_2$, & &  \\
             &  $y_2x_2 =-px_1y_1+p^2x_2y_1+x_1y_2-x_2y_2$. & & \\ \hline
\multirow{3}{*}{$\mathbb{D}$} &  $x_2x_1 = -x_1 x_2, \quad y_2y_1 = py_1y_2$, & & \\
             & $y_1x_1 = -px_1y_1,  \quad y_1x_2=-p^2x_2y_1+x_1y_2$, & $p\in \{-1,1\}$ & \checkmark \\
             &  $y_2x_1 = px_1y_2, \quad y_2x_2=x_1y_1+x_2y_2$. & & \\ \hline
\multirow{3}{*}{$\mathbb{E}$} &  $x_2x_1 = -x_1 x_2, \quad y_2y_1 = py_1y_2$, & & \\
             & $y_1x_1 = x_1y_2+x_2y_2,  \quad y_1x_2=x_1y_2-x_2y_2$, &  $p^2=-1$ &\checkmark \\
             &   $y_2x_1 = -x_1y_1+x_2y_1, \quad y_2x_2=x_1y_1+x_2y_1$. & & \\ \hline
\multirow{5}{*}{$\mathbb{F}$} &  $x_2x_1 = -x_1 x_2, \quad y_2y_1 = py_1y_2$, & &   \\
             & $y_1x_1 = -x_1y_1-px_2y_1+x_1y_2-x_2y_2$, & & \\
             & $y_1x_2 =-px_1y_1+x_2y_1+x_1y_2+x_2y_2$, & $p^2=-1$ & \\
             &  $y_2x_1 = -px_1y_1+px_2y_1+px_1y_2+x_2y_2 $, & & \\
             & $y_2x_2 =-px_1y_1-px_2y_1+x_1y_2-px_2y_2$. & &  \\ \hline
\multirow{5}{*}{$\mathbb{G}$} & $ x_2x_1 = x_1 x_2, \quad y_2y_1 = py_1y_2$, & & \\
             & $y_1x_1 = px_1y_1, $ &  $p\not = 0, \pm 1$ & \\ 
             & $y_1x_2=px_1y_1+p^2x_2y_1+x_1y_2$  &   and &\checkmark \\
             &  $y_2x_1 = px_1y_2, $ &  $\alpha \not = 0$ & \\
             & $y_2x_2=\alpha x_1y_1-x_1y_2+x_2y_2$. & & \\ \hline
    \end{tabular}
    \caption{Double extensions}
    \label{tab:1-DOE}
    \end{center}
\end{table}
\end{landscape}

\begin{landscape}
    \begin{table}[h]
    \centering
    \begin{tabular}{|c|c|c|c|}
\hline
Double extension & Relations defining the double extension & Conditions & GS bases finite \\ 
\hline
\multirow{3}{*}{$\mathbb{H}$} & $ x_2x_1 = x_1 x_2+x_1^2, \quad y_2y_1 = -y_1y_2$, & &  \\
             & $y_1x_1 = x_1y_2,  \quad y_1x_2=\alpha x_1y_2+x_2y_2$, & $\alpha \not = 0$ &\checkmark \\
             & $y_2x_1 = x_1y_1, \quad y_2x_2=\alpha x_1y_1+x_2y_1$. & & \\ \hline
\multirow{5}{*}{$\mathbb{I}$} & $ x_2x_1 = qx_1 x_2, \quad y_2y_1 = -y_1y_2$, & &  \\
             & $y_1x_1 = -qx_1y_1-qx_2y_1+x_1y_2-qx_2y_2,$ & & \\
             & $y_1x_2 = x_1y_1+x_2y_1+x_1y_2-qx_2y_2$, &  $q^2=-1$  & \\
             &  $y_2x_1 = x_1y_1+qx_2y_1+qx_1y_2-qx_2y_2,$ & &\\
             & $ y_2x_2 =-x_1y_1-qx_2y_1+x_1y_2-x_2y_2$. & &  \\ \hline
\multirow{3}{*}{$\mathbb{J}$} & $ x_2x_1 = qx_1 x_2, \quad y_2y_1 = -y_1y_2$, & &  \\
             & $y_1x_1 = x_2y_1+x_2y_2,  \quad y_1x_2=-x_1y_1+x_1y_2$, &  $q^2=-1$ & \\
             & $y_2x_1 = x_2y_1-x_2y_2, \quad y_2x_2=x_1y_1+x_1y_2$. & & \\ \hline
\multirow{3}{*}{$\mathbb{K}$} & $ x_2x_1 = qx_1 x_2, \quad y_2y_1 = -y_1y_2$, &  $q\in \{-1,1\}$ & \\
             & $y_1x_1 = x_1y_1,  \quad y_1x_2=x_2y_2$, &  and &\checkmark  \\
             & $y_2x_1 = x_1y_2, \quad y_2x_2=\alpha x_2y_1$. & $\alpha \not = 0$ & \\ \hline
\multirow{3}{*}{$\mathbb{L}$} & $ x_2x_1 = qx_1 x_2, \quad y_2y_1 = -y_1y_2$, & $q\in \{-1,1\}$ &  \\
             & $y_1x_1 = \alpha x_1y_2,  \quad y_1x_2=x_2y_2$ &  and &\checkmark \\
             & $y_2x_1 = \alpha x_1y_1, \quad y_2x_2=x_2y_1$. & $\alpha  \not = 0$  & \\ \hline
\multirow{3}{*}{$\mathbb{M}$} & $ x_2x_1 = -x_1 x_2, \quad y_2y_1 = -y_1y_2$, & &   \\
             &$y_1x_1 = x_2y_1+x_1y_2,  \quad y_1x_2=\alpha x_1y_1-x_2y_2$, & $\alpha \not = 1$ & \\
             & $y_2x_1 = x_1y_1-x_2y_2, \quad y_2x_2=-x_2y_1-\alpha x_1y_2$. & &  \\ \hline
\multirow{3}{*}{$\mathbb{N}$} & $ x_2x_1 = -x_1 x_2, \quad y_2y_1 = -y_1y_2$, & &  \\
             &$y_1x_1 = -\beta x_2y_1+\alpha x_2y_2,  \quad y_1x_2=\beta x_1y_1+\alpha x_1y_2$,  & $\alpha^2 \not = \beta^2$ &\\
             & $y_2x_1 = \alpha x_2y_1-\beta x_2y_2, \quad y_2x_2=\alpha x_1y_1+\beta x_1y_2$. & &  \\ \hline
\multirow{3}{*}{$\mathbb{O}$} &$ x_2x_1 = -x_1 x_2, \quad y_2y_1 = -y_1y_2$, & &  \\
             & $y_1x_1 = x_1y_1+\alpha x_2y_2,  \quad y_1x_2=-x_2y_1+x_1y_2$, & $\alpha \not = -1$ &\\
             & $y_2x_1 = \alpha x_2y_1-x_1y_2, \quad y_2x_2=x_1y_1+x_2y_2$. & &  \\ \hline
        \end{tabular}
         \caption{Double extensions}
    \label{tab:2-DOE}
    \end{table}
\end{landscape}

\begin{landscape}
\begin{table}[h]
    \centering
    \begin{tabular}{|c|c|c|c|}
\hline
Double extension & Relations defining the double extension & Conditions & GS bases finite \\ 
\hline
\multirow{3}{*}{$\mathbb{P}$} &$ x_2x_1 = -x_1 x_2, \quad y_2y_1 = -y_1y_2$, & &  \\
             & $y_1x_1 = x_1y_2+\alpha x_2y_2,  \quad y_1x_2=x_1y_2+x_2y_2$, & $\alpha \not = -1$ & \\
             &  $y_2x_1 = x_1y_1-\alpha x_2y_1, \quad y_2x_2=-x_1y_1+x_2y_1$. & &   \\ \hline
\multirow{3}{*}{$\mathbb{Q}$} &$x_2x_1 = -x_1 x_2,\quad y_2y_1 = -y_1y_2$, & &    \\
             & $y_1x_1 = x_1y_2,  \quad y_1x_2=x_1y_1+x_2y_1+x_1y_2$, & & \checkmark\\
             & $y_2x_1 = -x_1y_1, \quad y_2x_2=x_1y_1-x_1y_2+x_2y_2$. & &  \\ \hline
\multirow{3}{*}{$\mathbb{R}$} & $x_2x_1 = -x_1 x_2,\quad y_2y_1 = -y_1y_2$, & &  \\
             & $y_1x_1 = x_1y_1+x_2y_1+x_1y_2,  \quad y_1x_2=x_1y_2$, &  & \checkmark\\
             & $y_2x_1 = x_2y_1, \quad y_2x_2=-x_2y_1-x_1y_2+x_2y_2$. & &  \\ \hline
\multirow{5}{*}{$\mathbb{S}$} & $x_2x_1 = -x_1 x_2,\quad y_2y_1 = -y_1y_2$, & &  \\
             &$y_1x_1 = -x_1y_1+x_2y_1+x_1y_2+x_2y_2, $ & & \\
             &$ y_1x_2 = x_1y_1-x_2y_1+x_1y_2+x_2y_2$, & & \\
             & $y_2x_1 = x_1y_1+x_2y_1-x_1y_2+x_2y_2,$ & & \\
             &$y_2x_2 =x_1y_1+x_2y_1+x_1y_2-x_2y_2$. & &  \\ \hline
\multirow{5}{*}{$\mathbb{T}$} &$x_2x_1 = -x_1 x_2,\quad y_2y_1 = -y_1y_2$, & &  \\
             &$y_1x_1 = -x_1y_1+x_2y_1+x_1y_2+x_2y_2, $ & &\\
             &$ y_1x_2 = x_1y_1-x_2y_1+x_1y_2+x_2y_2$, & & \\
             &  $y_2x_1 = x_1y_1+x_2y_1+x_1y_2-x_2y_2, $ & & \\
             &$ y_2x_2 =x_1y_1+x_2y_1-x_1y_2+x_2y_2$. & &  \\ \hline
\multirow{5}{*}{$\mathbb{U}$} & $x_2x_1 = -x_1 x_2,\quad y_2y_1 = -y_1y_2$, & &  \\
             & $y_1x_1 = -x_1y_1+x_2y_1+x_1y_2+x_2y_2, $ & &\\
             &$ y_1x_2 = x_1y_1+x_2y_1+x_1y_2-x_2y_2$, & & \\
             & $y_2x_1 = x_1y_1+x_2y_1-x_1y_2+x_2y_2, $ & &\\
             &$ y_2x_2 =x_1y_1-x_2y_1+x_1y_2+x_2y_2$. & &  \\ \hline
\multirow{3}{*}{$\mathbb{V}$} & $x_2x_1 = x_1 x_2,\quad y_2y_1 = -y_1y_2$, & &  \\
             &$y_1x_1 = x_2y_1+x_1y_2,  \quad y_1x_2= x_2y_1$, & & \checkmark \\
             &  $y_2x_1 = -x_1y_1+x_2y_1, \quad y_2x_2=x_2y_2$. & &  \\ \hline
    \end{tabular}
    \caption{Double extensions}
    \label{tab:3-DOE}
\end{table}

\end{landscape}

\begin{landscape}
    \begin{table}[h]
        \centering
        \begin{tabular}{|c|c|c|c|}
          \hline
Double extension & Relations defining the double extension & Conditions & GS bases finite \\ 
\hline
\multirow{3}{*}{$\mathbb{W}$} & $x_2x_1 = x_1 x_2,\quad y_2y_1 = -y_1y_2$, & &  \\
             &$y_1x_1 = \alpha x_2y_1+x_1y_2,  \quad y_1x_2= x_1y_1-x_2y_2$, & $\alpha \not = -1$ & \\
             & $y_2x_1 = x_1y_1+\alpha x_2y_2, \quad y_2x_2=-x_2y_1+x_1y_2$. & &  \\ \hline
\multirow{3}{*}{$\mathbb{X}$} &$x_2x_1 = x_1 x_2,\quad y_2y_1 = -y_1y_2$, & & \\
             & $y_1x_1 = x_1y_2,  \quad y_1x_2= x_1y_2+x_2y_2$ &  & \checkmark \\
             & $y_2x_1 = x_1y_1, \quad y_2x_2=x_1y_1+x_2y_1$. & & \\ \hline
\multirow{3}{*}{$\mathbb{Y}$} & $x_2x_1 = x_1 x_2,\quad y_2y_1 = -y_1y_2$, & &  \\
             &$y_1x_1 = x_1y_1,  \quad y_1x_2= \alpha x_1y_1-x_2y_1+x_1y_2$, &$\alpha$ is general & \checkmark  \\
             &  $ y_2x_1 = x_1y_2, \quad y_2x_2=x_1y_1+\alpha x_1y_2-x_2y_2$. & &\\ \hline
\multirow{3}{*}{$\mathbb{Z}$} &  $x_2x_1 = -x_1 x_2,\quad y_2y_1 = y_1y_2$, & &  \\
             & $y_1x_1 = x_1y_1+x_2y_2,  \quad y_1x_2 = x_2y_1+x_1y_2$,& $\alpha(1+\alpha)\not = 0$ &  \\
             &  $y_2x_1 = \alpha x_2y_1-x_1y_2, \quad y_2x_2=\alpha x_1y_1-x_2y_2$. & is general & \\ \hline
        \end{tabular}
    \caption{Double extensions}
    \label{tab:4-DOE}
\end{table}
\end{landscape}

\end{document}